\documentclass[12pt A4paper]{amsart}

\usepackage{graphics, amsmath, amsfonts, amssymb, amscd, amsbsy,mathrsfs,color}

\newtheorem{theorem}{Theorem}[section]
\newtheorem{proposition}[theorem]{Proposition}
\newtheorem{lemma}[theorem]{Lemma}
\newtheorem{corollary}[theorem]{Corollary}

\theoremstyle{definition}
\newtheorem{definition}[theorem]{Definition}
\newtheorem{example}[theorem]{Example}
\newtheorem{Remark}[theorem]{Remark}
\newtheorem{Question}[theorem]{Question}

\newcommand{\B}{\mathbb{B}}
\newcommand{\N}{\mathbb{N}}
\newcommand{\C}{\mathbb{C}}
\newcommand{\D}{\mathbb{D}}

\newcommand{\R}{\mathbb{R}} 

\newcommand{\T}{\mathbb{T}}
\newcommand{\PSH}{\operatorname{PSH}}

\newcommand{\revi}[1]{{\color{black}{#1}}}
\DeclareMathOperator{\isdef}{\overset{def}{=}}

\begin{document}
\title[Boundary Pluripolar Hulls]{ Characterizations of Boundary Pluripolar Hulls}

\author{Ibrahim K. Djire }
\address{Jagiellonian University, Department of Mathematics}
\email{Ibrahim.Djire@im.uj.edu.pl}

\author{ Jan Wiegerinck  }
\address{Korteweg--de Vries Institute, Universiteit van Amsterdam, Science Park  105-107, Amsterdam} 
\email{J.J.O.O.Wiegerinck@uva.nl}

\maketitle

\begin{abstract}
We present some basic properties of the so called boundary relative extremal function and shed some light on Sadullaev's question about behavior of different kinds of extremal functions. We introduce and discuss \emph{boundary pluripolar} sets and   \emph{boundary pluripolar hulls}.   For B-regular domains the  boundary pluripolar hull is always trivial on the boundary of the domain. We present a ``boundary version''  of Zeriahi's theorem on the completeness of pluripolar sets.
\end{abstract}

\section{Introduction}

\revi{Boundary behavior of analytic functions in one or several complex variables is a classical subject, starting with the work of Fatou,  and the literature on it is so vast that it seems justifiable to omit references. The boundary behavior of harmonic and subharmonic functions is also classical and well understood, e.,g., ~\cite{ArGa}. The boundary behavior of plurisubharmonic functions, however, is less well understood. In this paper we mainly study properties of the \emph{boundary extremal function} $\omega(z,A, D)$, which is a generalization of the classical notion of harmonic measure. }

Throughout the paper $D$ will denote \revi{a bounded} domain in $\C^n$, $\PSH(D)$ the family of all plurisubharmonic functions on $D$, and $A$ a subset in the boundary of $D$.  For any function $u: D\rightarrow \R\cup \{-\infty\}$ and $x \in \overline D$  set   

$$u^*(x)=\limsup_{z\rightarrow x, z\in D} u(z)=\lim_{r\rightarrow 0}\sup_{B(x,r)\cap D } u,$$ 
the upper semicontinuous regularization of $u$ on $\overline D$.
We let $\D$ be the unit disc, $\T$ the unit circle and $\B$ the unit ball in $\C^2$.

Siciak, cf.~\cite{Si}, introduced the \emph{relative extremal  function} $\omega^*$, where $\omega$ is defined as follows. Given an open set  $D$ in $\C^n$ and a compact subset $E$ of $D$  $$\omega(z,E,D)=\sup\{u(z); u\in \PSH(D), u\leq 0, u\leq-1 \text{ on } E \}, \quad z\in D.$$ 

 \revi{Siciak's definition} also makes sense for subsets $A$ of $\partial D$. For $z\in D$ one defines, cf. \cite{Sa, Po, EdSi1} \revi{(note that  \cite{EdSi1} appeared as \cite{EdSi}, but in this latter paper there is little reference left to the boundary extremal function)},
 $$\omega(z,A,D)=\sup\{u(z):  u\in \PSH(D), u\le 0,  u^*\leq -1 \text{  on  } A \}.$$ 
 We  will  call $\omega^*(.,A,D)$ the \emph{boundary relative extremal function}. It is a special case of the (regularization of) the Perron-Bremermann function,\revi{ hence is always maximal in $D$, cf. \cite{Sa}.}
For a bounded function $f$ on $\partial D$ the Perron\revi{-}Bremermann function $u_f$ is defined as 
$$u_f=\sup\{v\in \PSH(D), v^*\leq f \text{ on } \partial D\}.$$

\revi{In \cite{Sa},  Sadullaev gave different versions of $\omega$. We will touch upon this in Section 2, where  we will study $\omega(.,A,D)$ somewhat further and give some additional properties and applications of it. 
}

\smallskip
   Following Sibony, cf. \cite{Sib}, we will say that a bounded domain $D\subset\C^n$ is \emph{B-regular} if every $f\in C(\partial D)$ can be extended  to a plurisubharmonic function on $D$ that is continuous on $\overline D$. In  \cite{Sib} it is proved that the following statements are  equivalent:
\begin{itemize}
 \item $D$ is B-regular;
 \item For $z\in\partial D$ there is $u\in \PSH(D)\cap C(\overline D)$ such that $u(z)=0$ and $u<0$ on $\overline D\setminus\{z\}$;
 \item There is $u\in \PSH(D)\cap C(\overline D)$ such that $\lim_{z\rightarrow\partial D} u(z)=0$ and $z\mapsto u(z)-|z|^2\in \PSH(D)$.
\end{itemize}
For  a B-regular domain $D$ \revi{and $f\in C(\partial D)$} we have $u_f\in \PSH(D)\cap C(\overline D)$ and $u_f=f$ on $\partial D$, cf. \cite{BeTa1}.

 For $A\subset \partial D,$ it can happen that any $u\in \PSH(D)$ such that $u^*|_A=-\infty$ assumes the value $-\infty$ automatically on a bigger set in $\overline D$. For instance, set $\B=\{(z_1,z_2)\in \C^2, |z_1^2|+|z_2^2|<1\}$. Let $A_1\subset\T$ be the closure of a half-circle. Set $A=A_1\times\{0\}$. Any $u\in \PSH(\B)$ such that $u^*\equiv -\infty$ on $A$ is identically $-\infty$ in $\{z\in\C, |z|<1\}\times \{0\}$. The phenomenon is similar to the occurrence of pluripolar hull \revi{$\hat E_D$ of a pluripolar subset $E$ of a domain $D$ in $\C^n$. This notion was introduced by Zeriahi in \cite{Ze}, and is defined as follows
\[\hat E_D=\{z\in D: \forall u\in \PSH(D)\ u|_E=-\infty \Rightarrow u(z)=-\infty\}.\]}
 
 We will call a subset $A\in \partial D$  \emph{b-pluripolar} (boundary pluripolar) if there exists  a $u\in \PSH(D),$ $u\leq 0$, $u\not\equiv -\infty$, such that $A\subset\{u^*=-\infty\}$  and  we will call a subset $A\in \partial D$ \emph{ completely b-pluripolar}  if there exists  a $u\in \PSH(D),$ $u< 0$,\quad  $u\not\equiv -\infty,$ such that $\{z\in \partial D,\quad u^*(z)=-\infty \}=A$. 
  Zeriahi showed in \cite{Ze} that if $E\subset D$ is pluripolar and an $F_\sigma$ as well as a $G_\delta$, then $E$ is \revi{completely pluripolar, i.e., there exists $u\in\PSH(D)$ with $E=\{z\in D: u(z)=-\infty\}$,}  if and only if $E$ coincides with its pluripolar hull. We will define the \emph{boundary pluripolar hull} in Definition \ref{Def3.2} and employ  $\omega(.,A,D)$ to describe this in Section 3 and 4. In Section 4 we will give a boundary version of Zeriahi's theorem. The proof is close to Zeriahi's.
We will  show  that for B-regular domains the b-pluripolar hull $\hat A\subset \overline D$ of a b-pluripolar set $A$ is contained in $A\cup D$. It is perhaps mildly surprizing that no hull is picked up at the boundary. In particular we have Corollary \ref{cor4.4} that for B-regular domains every b-pluripolar set that is $G_\delta$ as well as $F_\sigma$, is completely b-pluripolar.\\
  
In his thesis, \cite{Wi}, Wikstr\"om considered the function $V\in \PSH(\B)$:
$$V(z)= \log\frac{|z_2|^2}{1-|z_1|^2}$$
and observed that $V|_{\{z_2=0\}}=-\infty$ inside
 $\B$, but $V^*(z_1, 0)=0$ for $|z_1|=1$,\revi{ indeed, $V^*=0$ on all of $\partial \B$, }cf. \cite{Wi1}, Example 5.5. 
This example suggested to us that something like Corollary \ref{cor4.4} could hold.

\smallskip\noindent
\emph{Acknowledgement.} 
The first author is supported by the international PhD programme "Geometry and Topology in Physical Models" of the Foundation for Polish Science and he wishes to thank  Professor Armen Edigarian for his help to accomplish this work. \revi{The authors are grateful to an anonimous  referee for his suggestions  and comments, which have substantially improved the paper.}

\section{\revi{Properties of $\omega$}}

\begin{proposition}\label{prop2.1}
If $ A_1\subset A_2\subset\partial{D}$, then $$\omega(.,A_2,D) \leq \omega(.,A_1,D).$$
If $D_1\subset D_2$ and $A\subset \partial D_1\cap \partial D_2,$ then  on $D_1$ we have $$\omega(.,A,D_2) \leq \omega(.,A,D_1).$$
\end{proposition}

\begin{proposition} \label{prop2.7} Let $D\subset\C^n$ be a bounded domain and $A\subset\partial D$. 
Then
 $$\omega(.,A,D)=\sup_{A\subset V,\ V\text{\ open\ }} \omega(.,V,D),$$
and there is a non-increasing sequence $(V_j)_j$ of open neighborhoods of $A$ in $\partial D$ such that  
\begin{equation}\label{eq2.7}
 \lim_{j\rightarrow\infty}\omega(z,V_j,D)=  \omega(z,A,D),\quad \text{almost everywhere}.\end{equation}
\begin{proof}
Because of Proposition \ref{prop2.1} it suffices to prove the existence of a sequence $V_j$ such that \eqref{eq2.7} holds.
By Choquet's lemma there is an increasing sequence $(u_j)_j\subset \PSH(D)$ so that $u_j<0$ on $D$, $u^*_j\leq-1$ on $A$ and $$\lim_{j\rightarrow\infty} u_j=\omega(.,A,D),\quad \text{almost everywhere}.$$ 
Set $V_1=\{z\in\partial D, u^*_1-1<-1\}$ and $V_j=\{z\in\partial D, u^*_j-1/j<-1\}\cap V_{j-1}$ for $j>1$.
Obviously $u_j-1/j\leq\omega(.,V_j,D)$ for all $j$ hence we have (almost everywhere) 
$$\omega(.,A,D)= \lim_{j\rightarrow\infty}u_j \leq  \lim_{j\rightarrow\infty}  \omega(.,V_j,D)\leq \omega(.,A,D).$$
\end{proof}
\end{proposition}

\begin{proposition}\label{prop2.13}
Let $D$ be a domain in $\C^n,$ and let $A_1\supset A_2\supset A_3\supset \cdots $ be a sequence of compact subsets of $\partial D$ and $A=\cap_{j=1}^{\infty} A_j$.
Then at each point in $D$ $$\lim_{j\rightarrow\infty}\omega(.,A_j,D)=\omega(.,A,D),$$ 
\end{proposition}
\begin{proof}
Clearly, $\omega(.,A_1,D)\leq \omega(.,A_2,D)\leq \dots $  hence the limit exists. Take a negative function $v\in \PSH(D)$ such that $v^*|A\leq-1.$ As the set $V=\{z\in D: v(z)-\epsilon<-1\}$ is  open and $A$ is compact,  we can find an open set $U$ containing $A$ such that $U\cap D\subset V$. There exists $j_0$ such that  for each $j\geq j_0,$ $A_j\subset U.$ Therefore $v-\epsilon \leq \omega(.,A_j,D)$ for $j\geq j_0$. As a consequence, $v-\epsilon\leq \lim_{j\rightarrow\infty}\omega(.,A_j,D),$ and so $\omega(.,A,D)-\epsilon\leq \lim_{j\rightarrow\infty}\omega(.,A_j,D),$ this for all $\epsilon>0.$ The opposite inequality is trivial. 
\end{proof}

Edigarian and Sigurdsson, \cite{EdSi}, define a domain $D\subset\C^n$ to be \emph{weakly regular} if for every relatively open subset $U$ of $\partial D$ we have
\[\omega^*(.,U,D)\le -\chi_U\quad\text{on $\partial D$},\]
where $\chi_U$ is the characteristic function of $U$. They show that if  $D$ is Dirichlet regular in $\R^{2n}$ then $D$ is weakly regular. In particular hyperconvex domain are weakly regular. 

Note that if $D$ is weakly regular and $U$ is relatively open in $\partial D$, then $\omega(.,U, D)=\omega^*(., U, D)$, because $\omega^*(.,U, D)$ belongs to the family defining $\omega(., U, D)$. In particular $\omega^*(., U, D)$ is plurisubharmonic.


\begin{proposition}\label{prop2.3}
Let  $D\subset\C^n$ be weakly regular and $A\subset\partial D$. Then for all $x$ in the interior of $A$ we have  $$ \lim_{y\rightarrow x, y\in D}\omega^*(y,A,D)=-1.$$
\end{proposition}
\begin{proof}  Let $U=A^o$. Then for $y\in U$ we find
\[-1\le \omega^*(y, A, D)\le  \omega^*(y,U,D)\le -\chi_U(y)=-1.\]
\end{proof}

\begin{corollary}\label{cor2.4}
If $D\subset \C^n$ is a weakly regular domain and $A\subset \partial D$ is open, then $$\omega(.,A,D)=\omega^*(.,A,D).$$
Hence $\omega(.A,D)\in\PSH(D)$.
\end{corollary}




 \begin{proposition}  \label{prop2.6}
  Let $D$ be  a weakly regular domain in $\C^n$  and  $A \subset\partial D$ be open. Suppose that $\{D_j\}$ is an increasing sequence of open subsets of $D$ such that $D=\cup D_j$ and  $A\subset\cap_j\partial{D_j}.$ Then $$ \lim_{j\rightarrow \infty}\omega(x, A,D_j)=\omega(x, A,D),\quad \mbox{   for   } x\in D.  $$
 \end{proposition}
  \begin{proof}
 Set $v=\lim \omega(.,A,D_j)$. By Proposition \ref{prop2.1},  $\omega(.,A,D_{j+1}) \leq \omega(.,A,D_j)$ and because $A$ is relatively open,  $ \omega(.,A,D_j) \in \PSH(D_j)$, hence  $v\geq \omega(.,A,D)$. Now $v\in \PSH(D)$ and $v^*\leq-1$ on $A$, therefore $v\leq \omega(.,A,D)$. It follows that $v=\omega(.,A,D).$
\end{proof}

\begin{Remark} We don't know if the condition that $A$ be open, can be dropped.
\end{Remark}
  
\begin{proposition} \label{prop2.2} Let  $D\subset\C^n$ be B-regular  and $A\subset\partial D$. Then  $\omega^*(.,A,D)=0$ on $\partial D\setminus (\overline A)^o$.
\end{proposition}
\begin{proof} If $\partial D\setminus (\overline A)^o$ is empty there is nothing to prove, if not let $x\in \partial D\setminus (\overline A)^o$ and $r>0$. Let $z\in B(x,r)\cap\partial D\setminus \overline A$   and let $U$ be a neighborhood of $\overline A$ that does not contain $z$. Then there exists $f\in C(\partial D,[-1,0])$ such that $f=-1$ on $A$ and $f=0$ on $\partial D\setminus U$. We have $u_f\leq \omega(.,A,D)$. Thus $$ 0=u_f(z)\leq\sup _{B(x,r)\cap D} u_f\leq   \sup _{ B(x,r)\cap D}\omega(.,A,D)\leq 0.  $$
This holds for all $r>0$. Hence $0= \lim_{r\rightarrow 0}    \sup _{ B(x,r)\cap D}  \omega(.,A,D)=\omega^*(x,A,D)$.
\end{proof}

 


What the definition of ``right'' boundary behavior should be, is not entirely clear. 
In \cite{Sa} Sadullaev defines some alternative versions of the boundary extremal function, and in \cite {Wi, Wi1} Wikstr\"om considers for smoothly bounded domains $D$ as boundary value of $u\in\PSH(D)$ in a point $z\in \partial D$  the  value $u^R(z)=\limsup_{x\in N_z, x\to z}u(x)$ where $N_z$ is the real normal at $z$,  or $u^\alpha$ along a Kor\'anyi-Stein region at $z$. He shows by an example that the limit may depend on the aperture $\alpha$ of the region. These alternate definitions of boundary values in general do not lead to upper semicontinuous functions on the closure of the domain. 

We review Sadullaev's definitions, adapting the notation slightly to our situation. For the remainder of this section, $D$ will be a smoothly bounded domain.
Let $C_\alpha(\xi)=\{z\in D: |z-\xi|<\alpha\delta_\xi(z)\}$, where $\xi\in\partial D$, $\alpha> 1$, and $\delta_\xi(z)$ is the distance from $z$ to the real tangent space at $\xi$ to $\partial D$. Then for a function $u$ on $D$, set
\[u^\#(\xi)=\sup_{\alpha>1}\limsup_{z\to\xi,z \in C_\alpha(\xi)}u(z).\]

\begin{definition}[Sadullaev]
Let  $D$  and $A\subset\partial D$. For $z\in D$ set 

\begin{itemize}
\item  $\omega_1(z,A,D)=\sup\{ u(z), u\in \PSH(D)\cap C(\overline D), u\le 0, u|_A\leq-1\},$
\item  $\tilde \omega(z,A,D)=\sup\{ u(z), u\in \PSH(D), u\le 0, u^\#|_A\le -1 \},$
\item  $\omega_3(z,A,D)=\sup\{u(z), u\in \PSH(D) u\le 0, \limsup_{z\rightarrow\zeta, z\in n_{\zeta}} u(z)\leq-1 \text{ for all } \zeta\in A\}$, where $n_{\zeta}$ is the inward  normal to $\partial D$ at $\zeta$.
\end{itemize}
\end{definition}

\begin{Question}[Sadullaev, \cite{Sa}] Let $K\subset \partial D$ be compact. Clearly
\begin{equation}\label{eq-sa}\omega_1(z,K,D)\le \omega(z,K,D)\le\tilde \omega(z,K,D)\le \omega_3(z,K,D),\quad\text{for $z\in D$}.\end{equation}
For which $D$ and $K$ and which of these boundary extremal functions do the upper semicontinuous regularizations coincide? 
\end{Question}

As far as we know, little progress to this question has been reported.
If $D$ is B-regular, it is easy to see that $\omega_1(z,K,D)=\omega(z,K,D)$. Indeed, if $u\in\PSH(D)$ and $u^*\le -\chi_K$, then, because $u^*$ is upper semicontinuous, there exists a $v\in C(\partial D)$ with $v\ge u^*$ on $\partial D$ and $v=-1$ on $K$. Hence  $v_f\ge u$, and the equality follows.

For complete pseudoconvex Reinhardt domains, i.e., balanced multi-circular domains $D$ and multicircular $K$ we have the following result.
\begin{theorem}
Let $D$ be a smoothly bounded complete pseudoconvex Reinhardt domain in $\C^n$, and $K\subset\partial D$ a multi-circular compact set that does not meet any of the hyperplanes $\{z_j=0\}$.
Then all the inequalities in \eqref{eq-sa} are equalities. In particular this answers Sadullaev's question, all the upper semicontinuous regularizations coincide.
\end{theorem}
\begin{proof} Let $\Delta\subset \overline D$ be an open polydisc with distinguished boundary $T$ contained in $K$. By a linear change of variables, we can assume in the first part of the proof that $\Delta$ is the unit polydisc. 
Let $u\in\PSH(D)$, $u\le 0$ and $\limsup_{z\rightarrow\zeta, z\in n_{\zeta}} u(z)\leq-1 \text{ for all } \zeta\in K$. Replacing $u$ by $\max\{u, -1\}$, we can assume that the limsup is a limit and equals -1. From the multi-circular assumptions on $D$ and $K$, $n_\zeta=n_{|\zeta|}$, and because $D$ is pseudoconvex $n_\zeta$ is contained in $\Delta$ if $\zeta\in T$. Thus 
\[n_\zeta=\{(t+(1-t)r)\zeta\isdef ((t+(1-t)r_1)\zeta_1,\dots,(t+(1-t)r_n)\zeta_n):\ t\in[0,1]\}\]
for certain $0< r_1, r_2, \dots r_n<1$. 
Recall that plurisubharmonic functions are $n$-subharmonic, so that for all $w\in \Delta$ we find
\begin{equation}\label{eqPoi}\begin{split}u((t+(1-t)r)w)\le\quad\\
\le \frac1{(2\pi)^n}\int_T P(w_1,\zeta_1)\dots P(w_n,\zeta_n)u((t+(1-t)r)\zeta)\frac{d\zeta_1}{i\zeta_1}\dots\frac{d\zeta_n}{i\zeta_n}.\end{split}
\end{equation}
where $P$ denotes the Poisson kernel for the unit disc, cf.~\cite[Thm. 3.2.4]{Ru}. Because
\[\lim_{t\to 1}u((t+(1-t)r)\zeta=-1, \quad \zeta\in T,\]
the dominated convergence theorem gives that $u=-1$ on $\Delta$.

It follows that $u=-1$ on the open set $E=\cup \Delta$, where the union is taken over all polydiscs $\Delta$ with distinguished boundary in $K$.
Hence $u(z) \le \omega(z, E,D)$, and hence also $\omega_3(z, K, D)\le \omega(z, E, D)$. 

Let $\pi(z)=(\log|z_1|,\dots, \log|z_n|)$ and for $S\subset \C^n$, let $S_\R=\{x\in \R^n: x=\pi(z),\ z\in S\}$. Moreover, 
for $\Omega$ open and convex in $\R^n$ and $F\subset\Omega$ compact, put
\[A(x,F,\Omega)=\sup\{f(x):f\text{ convex on $\Omega$}, f\le 0, f|_F\le -1\}.\]
It is known, see \cite[Proposition 3.4.1]{JaPf} and also \cite{Tho} that
 \[ \omega(z, E, D)=A(\pi(z), E_\R,D_\R),\] 
and that $A(x,F,\Omega)$ is already obtained as the supremum of affine functions. It follows that $\omega(z,E,D)$ can be obtained as the supremum of  affine functions  of $\log|z_j|$, ($j=1,\dots,n$), that are bounded by 0 on $D$ and less then -1 on $E$. We may replace all such functions $u$ by 
$\tilde u=\max\{u, -1\}$. Because $\tilde u\in\PSH(D)\cap C(\overline D)$, we infer that $\omega_3(z,K, D)\le \omega_1(z,K, D)$, which proves the theorem.
\end{proof}

\section{Boundary pluripolar sets and boundary pluripolar hulls}
As in the classical case the boundary relative extremal function can be used to describe boundary pluripolar sets.
The characterizations of Sadullaev \cite{Sa}, Levenberg-Poletsky \cite{LePo}, also cf. \cite{EdSi1}, of pluripolar hulls and their proof also hold for b-pluripolar sets. We will include this result with its very similar proof for convenience of the reader in Proposition \ref{Prop3.3}. As in the classical case a countable union of b-pluripolar set is b-pluripolar (Proposition \ref{prop2.12}). However, in contrast with the classical case  where the relative extremal function $\omega^*(.,E,D)$ of a subset $E\subset D$ has the property that $\{z\in E, \omega^*(z,E,D)>-1\}$ is pluripolar, the set $\{z\in A, \omega^*(z,A,D)>-1\}$ is not in general b-pluripolar and the behavior of $ \omega^*(z,A,D)$ at the boundary of $D$ is not very informative, see Example \ref{Ex3.2}.

\noindent
\begin{definition}
We say  that a subset $A\in \partial D$ is a \emph{b-pluripolar set} if there exists  a $u\in \PSH(D)$, $u\leq 0$, $u\not\equiv -\infty$, such that $u^*=-\infty$ on $A$. 
\end{definition}

It is well known that a compact set $K\subset\T$ in the boundary of the unit disc $\D$ is b-polar if and only if it has arc length 0, and that not all such sets are polar. Hence there exist b-polar sets that are not polar. This example can be modified to the several variables situation.

\begin{example} Let $K$ be a b-polar set in $\T$ that is not polar and let $u$ be a subharmonic function on $D$ such that $u\le 0$ and $u^*|_K=-\infty$. 
Consider the function $v$ on the unit ball $\B\subset \C^2$ defined by $v(z,w)=u(z^2+w^2)$. Let 
$$A=\{(z,w)\in\partial \B: z^2+w^2\in K\}.$$
Then $v^*=-\infty$ on $A$, hence $A$ is b-pluripolar. Now if $A$ would be pluripolar we could find, invoking Josefson's theorem, cf.~\cite{Jo}, $f\in\PSH(\C^2)$ so that $f|_A=-\infty$. Consider for $\alpha\in[0,2\pi)$ the function $f_\alpha$ on $\C$ defined by $f_\alpha(\zeta)=f(\zeta\cos\alpha,\zeta\sin\alpha)$. It is subharmonic or identically equal to $-\infty$. Take a \revi{branch} $h(z)$ of $\sqrt z$ with branch cut not meeting $K$. Then $f_\alpha\circ h=-\infty$ on $K$. It follows that $f_\alpha\equiv-\infty$. In particular $f=-\infty$ on $\R^2\subset\C^2$, which is not a pluripolar set. The conclusion is that $A$ is not pluripolar.
\end{example}
\noindent
\begin{definition}\label{Def3.2}
Let  $ A\subset\partial D$ be b-pluripolar. The set  
$$ \{z\in \overline D:\  u^*(z)=-\infty, \text{ for all }u\in \PSH(D) \text{ with } u\not\equiv-\infty, \  u<0,\  u^*|_A=-\infty \}$$ 
will be called the \emph{b-pluripolar hull} of $A$ and  will be denoted by $\hat{A}$.
\end{definition}

\begin{example}\label{Ex3.2}
\revi{Let $A= A_\alpha=\{(e^{i\phi}\cos\alpha ,e^{i\psi} \sin\alpha ): \phi,\psi\in [0,2\pi)\}\subset\partial \B$ be the distinguished boundary of a polydisc $\Delta_\alpha$. } We have $ \omega^*(.,A,\B)\equiv0$ on $\partial \B$, see Proposition \ref{prop2.2}. But every $u\in \PSH(\B)$ such that $u^*|A\equiv-\infty$ is identically $-\infty$ on the polydisc, hence $u\equiv-\infty$ on $\B$ and $A$ is not b-pluripolar.

Similarly, for
$E_m=\cup_{j=1}^m A_{\alpha_j}$, we also find $\omega^*(.,E_m,\B)\equiv0$ on $\partial \B$. However, if we choose $(\alpha_j)_j$ a dense sequence in $(0,2\pi)$ we find for $z\in \partial \B$
\[0=\lim_{m\to\infty}\omega^*(z,E_m,\B)\ne \omega^*(z,\lim_{m\to\infty}E_m,\B)=-1.\]
Indeed, if $u\in \PSH(D)$ is negative and $u^*\le -1$ on all $E_m$ we have \revi{$u\le-1$} on $\cup_j \Delta_{\alpha_j}$.

\end{example}

\begin{proposition}[cf. \cite{Sa, LePo, EdSi1}]\label{Prop3.3}
Let $D\subset \C^n$ be   a domain in $\C^n$   and $A\subset \partial{D}$. Then the following conditions are equivalent :
\begin{enumerate}
\item $\omega^*(.,A,D)\equiv 0$;
\item $A$ is b-pluripolar.
\end{enumerate} 
In this case  
$$\hat{A} \cap D=\{z\in D, \quad \omega(.,A,D)<0\}.$$
\revi{In particular $\hat A\cap D$ is pluripolar.}
\end{proposition}

\begin{proof}
If $A$ is b-pluripolar,  take any $v\in \PSH(D)$, $v<0$, $v\not\equiv-\infty$ such that $v^*=-\infty$ on $A$. Then $\epsilon v\leq\omega(.,A,D)$ for all $\epsilon>0$,  hence if for some $z$ $\omega(z,A,D)<0$ then $v(z)=-\infty$, and it follows that $z\in \hat A$. Moreover, for all $z$ such that $v(z)>-\infty$ we find $\omega(z,A,D)=0$, hence $\omega^*(.,A,D)\equiv 0$.\\
Assume now that $\omega^*(.,A,D)\equiv 0$. Let $z\in D$ be such that $\omega(z,A,D)=0$.  For $j\in \mathbb N$ there is a negative $u_j\in\PSH(D)$ with $u_j^*|_A\le -1$ and  $u_j(z)>-2^{-j}$.  Define 
$$v(y)=\sum_{j=1}^{\infty} u_j(y)  \quad (y\in D). $$
Observe that $v(z)>-1$, hence as a limit of a decreasing sequence of negative plurisubharmonic functions, $v\in \PSH(D)$, negative and not identically $-\infty$. Moreover, $v^*|A\equiv-\infty$. We conclude that $A$ is b-pluripolar and $z\not\in \hat A$. Finally, as $\hat A\cap D=\{z\in D: \omega(z, A, D)\ne\omega ^*(z, A, D)\}$ is negligible, it is pluripolar.
\end{proof}

 \begin{proposition}\label{prop2.12}
Let $D$ be a bounded domain in  $\C^n$. Suppose that $A=\cup_j A_j,$ where $A_j\subset \partial D$ for $j=1,2,\cdots $. If $\omega^*(., A_j,D)\equiv 0$ for each $j$, then $\omega^*(., A,D)\equiv 0$,
\revi{hence a countable union of b-pluripolar sets is b-pluripolar.}
\end{proposition}

\begin{proof}
By Proposition \ref{Prop3.3} above, we can choose $v_j\in \PSH(D)$ such that $v_j<0$ on $D$  and $v_j^*|A_j\equiv-\infty$. Take a point $a\in (D\setminus\cup_jv_j^{*-1}(\{-\infty\}))$. By multiplying each of the functions $v_j^*$ by a suitable positive constant, we may suppose that $ v_j(a)>-2^{-j}.$ As in the proof of the proposition above we check that $v=\sum_j v_j \in \PSH(D), $ $v<0$ on $D$ ,  $v \not\equiv -\infty$ on $D$ and $v^*=-\infty$ on $A.$ By the previous proposition , $\omega^*(.,A,D)\equiv 0.$
\end{proof}

\begin{proposition}\label{Prop3.2}
Let $D\subset \C^n$ be a  B-regular domain  and let $A\subset\partial{D}$ be b-pluripolar. Then 
$$ \hat{A} \cap \partial D= A. $$
\end{proposition}

\begin{proof}
Obviously $ A\subset \hat A\cap\partial D.$  Now let $z\in\partial D\setminus  A$.  As $A$ is b-pluripolar  there exists $u\in \PSH(D)$ such that $u<-1$, $u^*=-\infty$ on $A$.  If $u^*(z)$ is finite, there is nothing to prove. We will assume $u^*(z)=-\infty$ and construct a function $v\in\PSH(D)\cap  C(\overline D\setminus\{z\})$ so that $(u+v)^*(z)$ is finite \revi{and so that $u+v$ is negative in $D$}. This then shows that $z\notin \hat A$. Let 
$$E_z(j)=\{w\in\partial D: \ \frac{1}{4j+1}\le |z-w| \le \frac1{4j}\}.$$
 Because $u^*$ is usc on $\partial D$ and $A$ is b-pluripolar, \revi{while $E_z(j)$ is not b-pluripolar, $u^*$ assumes a maximum $M_j$ on $E_z(j)$ with $-\infty<M_j\le -1$,} say in $w_j\in E_z(j)$. Let $f_j\le 0$ be continuous on $\partial D$, $f_j>u^*$ and $f_j(w_j)<u^*(w_j)+1$ and let $0\le \chi_j\le 1$ be a smooth function  on $\partial D$ with $\chi_j(w_j)=1$ and compactly supported in $\frac{1}{4j+2}\le |z-w| \le \frac1{4j-1}$, and of sufficiently small size to be determined later. Then 
\begin{align} \label{eq3.2} u^* &\le \sum_{j=1}^\infty  f_j\chi_j \quad \text{on $\partial D$}.\\
\intertext{and}
\label{eq3.3}u^*(w_j)&\ge \sum_{j=1}^\infty  f_j\chi_j (w_j)-1\quad \text{for every $j$}.
\end{align}
 Let $F_j$ be the harmonic function on $D$, continuous on $\partial D$ with boundary values $ -f_j\chi_j$. The series $\sum_{j=1}^\infty F_j$ represents a monotonically increasing sequence of harmonic functions that are continuous up to $\partial D$. By choosing the support of $\chi_j$ sufficiently small, we can achieve, in view of Harnack's theorem, that the series converges uniformly on compact sets in $\overline D\setminus \{z\}$ and represents a harmonic function on $D$ that is continuous on $\overline D\setminus \{z\}$ and has boundary values $\sum_{j=1}^\infty -f_j\chi_j$. 

Now let $v_j= u_{F_j}$ be the  Perron-Bremermann function of $-f_j\chi_j$. Then  $0\le v_j=v_j^*\le F_j$ on $\overline D$ with equality on $\partial D$ because $D$ is B-regular, and $v_j$ is a continuous plurisubharmonic function. It follows that the series $v=\sum_{j=1}^\infty v_j$ is also uniformly convergent on compact sets in $\overline D\setminus\{z\}$, hence it represents a plurisubharmonic function that is continuous up to $\partial D\setminus \{z\}$ with boundary values $\sum_{j=1}^\infty F_j$ on $\partial D\setminus\{z\}$. Then by \eqref{eq3.2} and \eqref{eq3.3} we have
\[u^*+v=\lim_{k\to\infty} (u^*+\sum_{j=1}^kv_k)\le 0\quad\text{\and}\quad \revi{(u^*+v)(w_j)\ge -1 }\text{\ for all $j$}.\]
Because $u^*+v^*$ is usc, we have that $(u^*+v^*)(z)\ge -1$.
\end{proof}

\begin{theorem}\label{thm3.1}
Let $D\subset \C^n$ be   B-regular  and  $A\subset \partial{D}$ be a  b-pluripolar set. Then 
$$\hat{A} =  A\cup \{z\in D, \quad \omega(z,A,D)<0\}.$$
\end{theorem}
\noindent

\begin{proof}Combine Proposition \ref{Prop3.3} and Proposition \ref{Prop3.2}.
\end{proof}

\begin{Remark}
Of course, if the domain is not B-regular, Proposition \ref{Prop3.2} is no longer valid.  \revi{Take for  $D$ the standard open polydisc} and let $A=\{(z,1), |z|=1\}$, then $A$ is b-pluripolar, which is seen by considering $\log|w-1|$ and $\hat A=\{(z, 1), |z|\le 1\}$. The same applies for domains with (fine) analytic discs in the boundary, cf.~\cite{Sib}.
\end{Remark}

\section{Completeness of b-pluripolar sets} 

\begin{definition}
We say  that a subset $A\in \partial D$ is \emph{completely b-pluripolar}  if there exists  a $u\in \PSH(D),$ $u< 0$,\quad  $u\not\equiv -\infty,$ such that $\{z\in \partial D,\quad u^*(z)=-\infty \}=A.$ 
\end{definition}
Zeriahi,\cite{Ze} gave conditions under which a pluripolar set is completely pluripolar. Here we  adapt  Zeriahi's result to boundary pluripolar sets. Our result requires only minor adaptations.

\begin{proposition}\label{prop4.1}
Let $D\subset \C^n$ be a B-regular domain  and $A\subset \partial D$ be a b-pluripolar set. Suppose that $F$ and $K$ are compact subsets  of $\overline D$ with $F\subset \hat A$ and  $K\subset \overline D\setminus \hat A.$ Then for all $C>0$ there exists $\psi_K\in \PSH(D)\cap C(\overline D)$ so that $\psi_K<0$,  $\psi_K< -C$ on $F$, and $\psi_K> -1$ on $K.$ \end{proposition}
\begin{proof}
Let $a\in K\subset \overline D\setminus \hat A$. Then there exists a negative  $u\in\PSH(D)$  so that $u^*=-\infty$ on $\hat A$ and $u^*(a)>-\infty$. Set $M=\sup\{u^*(z)-u^*(a), z\in  \overline  D \}$.  Then 
$$ w(z)=\frac{u(z)-u^*(a)}{2(|M|+1)} -1/2,\quad \text{ for  $z\in D$},$$
is  plurisubharmonic and $w<0$ on $D$, $w^*|\hat A=-\infty$, $w^*(a)=-1/2.$  
By ~\cite{Wi1}, Theorem 4.1,  we can find a sequence  in $\PSH(D)\cap C(\overline D)$ that decreases to $w^*$ on $\overline D$. In particular there exists in view of Dini's theorem a negative $f_a\in \PSH(D)\cap C(\overline D)$ such that $f_a<-C$ on $F$ and
$f_a(a)\geq w^*(a)=-1/2>-1 $. Then there exists a neighborhood $V_a$ of $a$ so that  $f_a(z)>-1$  for all $z\in V_a$. By compactness  we can find a finite subset  of $I\subset K$ such that $K\subset \cup_{a\in I} V_a$ . Set $\psi_K=\max\{f_a, a\in I\}$ then 
$$\psi_K<0,\ \ \psi_K\in \PSH(D)\cap C(\overline D),  \ \ \psi_K<-C \ \text{ on $F$},\ \ \text{ and } \psi_K>-1\ \text{ on $K$}.$$
\end{proof}

\begin{lemma}\label{lem4.2} Let $D$ be \revi{a domain} in $\C^n$.
Let $A\subset\partial D$ be b-pluripolar and let $ K\subset\partial D\setminus A$ be compact. Then there exists an $L\subset D\setminus\hat A$ such that   \revi{  $K\subset \overline L$ } and $L\cup K$ is compact. 
\end{lemma}
\begin{proof} \revi{If $K$ is empty there is nothing to prove.}
As $ K$ is compact there exist for every $j\in \N$ $N_j$ points $z_{jl}\in K$, $1\le l\le N_j$ such that $K\subset \cup_{l=1}^{N_j}B(z_{jl},1/j)$.
\revi{Because of Proposition \ref{Prop3.3}  $\hat A$ has empty interior and we can find a point $w_{jl}\in (B(z_{jl},1/j)\cap D)\setminus \hat A$.} Now let $L=\{w_{lj}: 1\le l\le N_j, j\in\N\}$. Then the limit points of $L$ belong to $ K$ hence $K\cup L$ and if $z\in K\cap B(z_{lj}, 1/j)$ then $|z-w_{lj}|<2/j$, therefore $z$ is a limit of a subsequence of $L$. 
\end{proof}

\begin{theorem} \label{thm4.3} Let $D$ be a B-regular domain in $\C^n$.
Let $A\subset\partial D$ be b-pluripolar, $F$ an $F_{\sigma }$ set, $G$ a $G_{\delta}$ set in $\overline D$ such that $F\subset \hat A\subset G$. Then there exists an  $ E\subset\overline D$  and a negative function $\psi\in\PSH(D)$  such that  $F\subset E\subset G$, where $E=\{z\in\overline D: \ \psi^*(z)=-\infty\}$.
\end{theorem}

\begin{proof}\revi{Let $f\in\PSH(D)$ be negative, $f^*|_A=-\infty$ and $f\not\equiv -\infty$. Then $G_0=\{f^*=-\infty\}\subset\overline D$ is a $G_\delta$ containing $\hat A$ with the property that $\partial D\setminus G_0$ and $D\setminus G_0$ are non-empty. Replacing $G$ by $G\cap G_0$, we have $\partial D\setminus G$ and $D\setminus G$ are non-empty.}
Set $F=\cup_j F_j$ where $(F_j)_{j\geq 1}$ is an increasing sequence of compact sets in $\hat A$, and $ \overline D\setminus G= \cup_j \tilde K_j$ where $(\tilde K_j)_j$ is an increasing sequence of  compact sets in $ \overline D\setminus G.$  Applying Lemma \ref{lem4.2} to $\tilde K_j\cap\partial D$, each $\tilde K_j$ can be enlarged to a  compact set $K_j\subset \overline D\setminus \hat A$ with the property that $K_j\cap\partial D\subset \overline{K_j\cap D}$. Replacing $K_{j+1}$ by $K_{j+1}\cup K_j$ if necessary, we can assume $K_j\subset K_{j+1}.$  
By  Proposition \ref{prop4.1} for each $j>0$ there exists  $\psi_j\in \PSH(D)\cap C(\overline D)$ with  
\begin{equation}\label{eq4.3} \psi_j\leq -2^{j}\;\text{ on $F_j$}, \quad \text{ and } \quad \psi_j\geq-1\;\text{ on  $K_j$}.
\end{equation}
The function $\psi=\sum_{j=1}^{\infty}2^{-j}\psi_j$  is negative. For $z\in D\setminus G$ there is $J>0$ so that $z\in K_J$ and we find that  
\begin{equation}\label{eq4.4} \revi{ \psi(z)=\sum_{j=1}^{J}2^{-j}\psi_j(z) +  \sum_{j=1+J}^{\infty}2^{-j}\psi_j(z)\geq \inf_{K_J}     \sum_{j=1}^{J}2^{-j}\psi_j-1>-C_J>-\infty,}
\end{equation}
where $C_J$ depends only on $K_J$, in view of the continuity of the $\psi_j$.
It follows that $\psi$ is plurisubharmonic on $D$ as limit of a decreasing sequence of plurisubharmonic functions. It satisfies $\psi^*\equiv-\infty $ on $ F$ because of \eqref{eq4.3}. Finally if $z\in \partial D\setminus G$, then $z\in \overline{K_j\cap D}$ for some $j$ and by \eqref{eq4.4}  $\psi^*(z)>C_j$,
hence $\psi^*>-\infty$ on $\overline D\setminus G$.
Set $E=\{z\in \overline D, \psi^*(z)=-\infty\}$ then $F\subset E\subset G$.
\end{proof}

\revi{
\begin{Remark} In Proposition \ref{prop4.1} and Theorem \ref{thm4.3} we only used B-regularity to the effect that negative $u\in\PSH(D)$ have the property that $u^*$ can be approximated on $\overline D$ by a decreasing sequence of functions in $\PSH(D)\cap C(\overline D)$. Domains with this \emph{approximation property}  were studied in \cite{Wi1}, where it is shown that B-regular domains and polydiscs have the approximation property. To our knowledge there are no other examples known.
\end{Remark}}

\begin{corollary}\label{cor4.4}
\revi{Let $D$ be a B-regular domain.} Every  b-pluripolar set $A\subset \partial D$ that is a $G_\delta$ as well as an $F_\sigma$ is completely b-pluripolar.
If, moreover, $\hat A$ is a $G_\delta$, then $\hat A=\{z\in \overline D: \psi^*(z)=-\infty\}$. In particular $\hat A\cap D$ is completely pluripolar.
\end{corollary}
\begin{proof}
By Proposition \ref{Prop3.2} $\hat A\cap \partial D= A$. We apply Theorem \ref{thm4.3} with $F=A$ and $G=A\cup D$. The theorem gives us a negative $\psi\in \PSH(D)$ with $A=\{z\in \partial D \text{ with } \psi^*(z)=-\infty\}$. In particular, $\psi\not\equiv-\infty$ on $D$, \revi{because it has finite boundary values  on $\partial D\setminus A$, and hence $A$ is completely b-pluripolar. 

If $\hat A$ is moreover a $G_\delta$, we apply the Theorem with $F=A$, $G=\hat A$ and obtain a function $\psi$ such that 
\begin{equation}\label{Dji} A\subset \{\psi^*=-\infty\}\subset \hat A.\end{equation}
Now $\psi^*|_A=-\infty$ implies $\hat A\subset \{\psi^*=-\infty\}$, hence the last inclusion in \eqref{Dji} is an equality.}
\end{proof}

\end{document}